\documentclass{aims}
\usepackage[english]{babel}
\usepackage{amssymb}                 
\usepackage{amsthm}                         
\usepackage{color}      
\usepackage{amsmath}
  \usepackage{paralist}
  \usepackage{graphics} %% add this and next lines if pictures should be in esp format
  \usepackage{epsfig} %For pictures: screened artwork should be set up with an 85 or 100 line screen
\usepackage{graphicx}  \usepackage{epstopdf}%This is to transfer .eps figure to .pdf figure; please compile your paper using PDFLeTex or PDFTeXify. 
 \usepackage[colorlinks=true]{hyperref}
\hypersetup{urlcolor=blue, citecolor=red}

\def\currentvolume{X}
 \def\currentissue{X}
  \def\currentyear{200X}
   \def\currentmonth{XX}
    \def\ppages{X--XX}

\newcommand{\haz}{\widehat}

\newcommand{\QQQ}{\haz{\mathcal{Q}}}
\newcommand{\R}{{\mathbb R}}

\newcommand{\Stbl}{\mathcal{S}}

\newcommand{\msat}{m_{\textrm{sat}}}

\newcommand{\md}{{\rm d}}
\renewcommand{\O}{\Omega}

\def\Gliminf{\mathop{\Gamma\text{--}\mathrm{liminf}}}

\newcommand{\epsi}{\varepsilon}

\newcommand{\EEE}{\color{black}}

\newcommand{\UU}{\mathcal U}
\newcommand{\MM}{\mathcal M}
\newcommand{\KK}{\mathcal K}
\newcommand{\QQ}{\mathcal Q}

\newcommand{\ds}{{\rm d}s}

\newtheorem{theorem}{Theorem}[section]

\newtheorem{lemma}[theorem]{Lemma}
\newtheorem{proposition}{Proposition}

\theoremstyle{definition}
\newtheorem{definition}[theorem]{Definition}

\title[Quasistatic magnetoelastic thin film]
      { Quasistatic  evolution of magnetoelastic \\ thin films via dimension
reduction}

\author[Martin Kru\v{z}\'{i}k, Ulisse Stefanelli, and Chara Zanini]{}


 \email{kruzik@utia.cas.cz}
 \email{ulisse.stefanelli@univie.ac.at}
 \email{chiara.zanini@polito.it}

\thanks{Martin Kru\v{z}\'{i}k is partially supported by projects  P201/10/0357 and 14-15264S (GA\v{C}R)}

\begin{document}
\maketitle

% Enter the first author's name and address:
\centerline{\scshape Martin Kru\v{z}\'{i}k}
\medskip
{\footnotesize
% please put the address of the first author
 \centerline{Institute of Information Theory
    and Automation of the ASCR}
   \centerline{Pod Vod\'{a}renskou v\v{e}\v{z}\'{i}
    4, 182 08 Prague, Czech Republic}
\centerline{and}
   \centerline{Faculty of Civil Engineering,
    Czech Technical University}
\centerline{Th\'{a}kurova 7, 166 29 Praha 6, Czech
    Republic} 
} % Do not forget to end the {\footnotesize by the sign }
\bigskip

\centerline{\scshape Ulisse Stefanelli}
\medskip
{\footnotesize
% please put the address of the first author
 \centerline{Faculty of Mathematics, University of Vienna}
   \centerline{Oskar-Morgenster-Platz 1, 1090 Vienna, Austria}
\centerline{and}
 \centerline{Istituto di Matematica Applicata e Tecnologie
   Informatiche {\it E. Magenes}, CNR}
   \centerline{via Ferrata 1, 27100 Pavia, Italy}
} % Do not forget to end the {\footnotesize by the sign }
\bigskip

\centerline{\scshape Chiara Zanini}
\medskip
{\footnotesize
% please put the address of the first author
 \centerline{Dipartimento di Scienze Matematiche
{\it G. L. Lagrange}, Politecnico di Torino}
\centerline{Corso Duca degli Abruzzi 24, 10129 Torino, Italy} 
} % Do not forget to end the {\footnotesize by the sign }

\bigskip

% The name of the associate editor will be entered by an editorial staff
% "Communicated by the associate editor name" is not needed for special issue.
 \centerline{(Communicated by the associate editor name)}

%The abstract of your paper
\begin{abstract}
 A rate-independent model for the  quasistatic  evolution of
a magnetoelastic  thin film is advanced  and
analyzed. Starting from the  three-dimen\-sional setting, we present
an evolutionary
  $\Gamma$-convergence argument  in order to  pass to
the limit  in  one of the material dimensions. By  taking
into account both  
conservative and dissipative actions, a nonlinear evolution system of rate-independent
type  is  obtained. The existence of so-called {\it energetic
  solutions} to  such  system  is proved via
approximation. 
\end{abstract}

%The title of your section 1
\section{Introduction}
Magnetoelasticity (or magnetostriction)  is the property of certain solids exhibiting
a strong coupling between mechanical and magnetic variables. As effect
of this coupling, relevant reversible mechanical deformations can
be induced by the application of an external magnetic field. This
behavior is clearly of a great applicative interest in connection
with sensors and actuators design, as well as for a variety of innovative functional-material devices. 

 The origin of magnetoelasticity lies in the interplay between
material crystallographic patterning (where different crystals present
different easy axis of magnetization) and magnetic domains. In absence
of external magnetic fields, magnetic domains orient in such a way to
minimize long-range dipolar effects. This generically results in some
small or even negligible magnetization of the medium. Upon applying an  external %eternal
magnetic field the magnetic
domains tend to reorient toward it. As  magnetizations  are
related to specific stress-free reference strains, this causes indeed
the emergence of a macroscopic deformation. As the intensity of the 
magnetic field is increased, more and more 
magnetic domains orientate themselves so that 
their principal axes of anisotropy are collinear 
with the magnetic field in each region and finally 
saturation is reached. We refer to e.g.~\cite{brown} for  a
discussion of the  physical
foundations of  magnetoelasticity.  

The mathematical modeling of magnetoelasticity is a vibrant
area of research, in particular, in connection with intelligent
materials  such as iron/rare-earth giant magnetostrictive materials \cite{DKV,james-kinderlehrer,K9} and
magnetic shape-memory alloys \cite{bks,bs,JS}. Correspondingly, the
understanding of the statics of these materials has attracted
considerable attention
\cite{DD98,desimone-james,rybka-luskin}. Building  from one
side   upon the static thin-film-limit analysis for
magnetic materials in {\sc Gioia \& James} \cite{goia-james} and 
from the other side   on the dimension reduction for static linear elastic plates
in \cite{freddi-paroni-zanini}, we shall be here concerned with the evolutive situation instead. In
particular,   we are interested in a slow  quasistatic  evolution
of such materials  under the  combined  action of conservative and
dissipative forces. We indeed   assume that  the  change of
magnetization  implies 
dissipation while  no dissipation is  associated with elastic
variables. The evolution is driven by a Dirichlet boundary condition
and/or by an external magnetic field.  Changes of external
conditions are considered to be  slow enough so that  inertial effects can be
neglected and  the system is always in 
equilibrium  so that its evolution is  quasistatic.   

 The focus of the paper is on deriving a  quasistatic  evolution
theory for magnetoelastic thin films. After considering the
above-mentioned dissipative evolution problem for the bulk,
three-dimensional material, we address the thin-film evolution by
means of a dimension reduction argument. In particular, we assume that
the reference configuration of the body is thin in one dimension and
we pass to the limit with respect to it. Our 
specific choice for the  scaling  entails that  in  the limit
one obtains a  Kirchhoff-Love   quasistatic  evolution  plate model
 for  magnetoelastic materials. Moreover, we are able to deal
 with thickness-depedent change of the anisotropic magnetic behavior
 of the sample, both at the static and the evolutive level. 

 The novelty of the paper is indeed twofold. From the one hand, we
advance the first  quasistatic  evolution model for
magnetoelastic thin films and prove the existence of suitable
variational solutions.  In particular, the emergence of the so-called
magnetic anisotropic behavior is emphasized.  Secondly, by deriving such a model
by dimension reduction, we provide a novel evolutive
$\Gamma$-convergence result in the magnetoelastic context.
This consists in combining some slightly refined
version of the already available static thin-film-limit theory within the general
frame of the evolutive
$\Gamma$-convergence analysis for rate-independent systems from
\cite{mrs}. We shall observe that, as already in \cite{goia-james},
 the magnetostatic energy contribution which is  usually 
difficult to evaluate in micromagnetics reduces in our model to
calculating the square of the third component of the
magnetization. This makes  the   model attractive from the
point of view of numerics. 

 Apart from the magnetoelastic setting,  
dimension reduction via $\Gamma$-convergence in the  quasistatic  evolutive setting has already
attracted some attention.
{\sc Liero \&
  Mielke} derive in \cite{ML11,M12} an elastoplastic plate theory in
presence of linear kinematic hardening. A
different theory is then obtained by an alternative scaling choice by
{\sc Liero \& Roche} \cite{LR12}. The perfectly-plastic case has then
be considered by {\sc Davoli \& Mora} \cite{DM13} and {\sc
  Davoli} \cite{D13,D14}, also in the frame of finite
plasticity. 
 {\sc Babadjian} obtained in \cite{Ba06} via dimension reduction the existence of a  quasistatic  evolution for a free crack in an elastic brittle thin film.
Dimension reduction in a delamination context is addressed
in \cite{FPRZ,FRZ13} whereas an application to shape-memory thin films
is described in \cite{BKP}. 

The plan of the paper is as follows.  We start by describing the bulk
model in the static three-dimensional situation in Section
\ref{description}. Then, the
corresponding static thin-film micromagnetic and magnetoelastic limits are
discussed in Section~\ref{dimension-reduction}. Eventually, Section
\ref{evo} focuses on  quasistatic  evolution situations both in the
bulk (Subsection~\ref{analysis}) and in the thin-limit case
(Subsection \ref{KL}, respectively). In particular, Subsection
\ref{KL} contains our main  convergence result,  i.e.  Theorem~\ref{thm}.

\section{ Description of the static bulk model}\label{description}
\setcounter{equation}{0}

 Let us start by specifying the modelization in the
three-dimensional setting. The thin-film model will then be derived in
Section \ref{evo} by means of a rigorous dimension reduction procedure. We assume to have fixed 
 an orthonormal  basis $\{e_1,e_2,e_3\}$  of $\R^3$ and  to be given  a thin
 magnetic   body with reference configuration 
 $\Omega_h:=\{(x_1,x_2)\in S;\ 0<x_3<h\}$.  Here, $S\subset\R^2$ 
 is  a bounded Lipschitz domain in the   $\{e_1,
 e_2\}$-plane  and $h >0 $ represents the small thickness of the
  specimen, eventually bound to go to $0$.  

 \subsection{Micromagnetics}
The magnetization of the body is described by  $\bar
 m:\Omega_h\rightarrow \R^3$ subject to the  saturation  constraint
$$|\bar m(x)|=\msat  \ \text{ for a.e. } x\in\Omega_h,$$
where the {\it saturation magnetization} $\msat>0$ is assumed to be
constant.  The  micromagnetic energy of the film  is
classically defined as \cite{brown,desimone-james,kruzik-prohl} 
\begin{align}
\bar E^{\rm mag}_h(t,\bar m)&:=\frac{1}{|S|h} \int_{\O_h}\left(\alpha
  |\nabla \bar m(x)|^2{+}\bar \varphi_h  (x,  \bar m(x)){+}\frac12\bar m(x)\cdot\nabla
   \bar \xi  \right) \md x \nonumber\\
& - \int_{\O_h} \bar H(x)\cdot \bar m(x) \md x. \label{micromagenergy}
\end{align}
 The {\it stray field} $\nabla \bar \xi$ is related to 
the magnetization  $\bar m$ via the  Maxwell  equation
\begin{equation}
\nabla\cdot(-\mu_0\nabla \bar \xi+\bar m\chi_{ \O_h})=0 \ \mbox{ in }\R^3\nonumber
\end{equation}
 where  $\mu_0$ is the vacuum permeability  and
$\chi_{ \O_h}$ is the {\it characteristic function} of the domain
$\O_h$,  namely $\chi_{ \O_h}=1$ on $\O_h$ and $\chi_{ \O_h}=0$
elsewhere in $\R^3$. 

The first term in the integral in (\ref{micromagenergy}) is the {\it exchange
energy}, penalizing indeed spatial changes of the
magnetization. 

The (thickness-dependent) {\it magnetic potential} 
$\bar \varphi_h  : \Omega_h \times \msat S^2 \to [0,\infty)$ describes the
magnetic anisotropy of the material. In particular, for all
thicknesses $h>0$ and $x \in \Omega_h$  it is an even function 
vanishing precisely at the set 
$\{\pm s_i;\,|s_i|=\msat\}_{i=1}^N$  for some $s_i=s_i(x)$,  where
$N=1$ for uniaxial magnets and $N=3$ or $N=4$ for cubic magnets. 
The lines through $\pm s_i$ are called {\it easy axes} of
the magnet. The space dependence in $\bar \varphi_h$ is intended
to model the polycrystalline texture of the medium and we assume $\bar
\varphi_h$ to be continuous. This particularly entails that the anisotropic energy
term is lower semicontinuous with respect to the $L^2$ topology.

 Following the classical theory by N\'eel
\cite{Nell54}, we allow the magnetic anisotropy of the medium to depend
on the sample thickness. It is indeed observed that many material
systems develop a very strong magnetic anisotropy in the off-plane
direction as $h \to 0$, see \cite{Bruno88,Schulz94}, for instance. This effect
is at the basis of the so-called {\it perpendicular recording}
technology, see the review \cite{Richter07}. \EEE 

The term containing $\nabla \xi$ is the so-called {\it
  stray-field energy} and represents long-range
dipolar self-interactions  favoring indeed the formation of a 
solenoidal  magnetic field. In particular, $\xi$ is the {\it
  magnetostatic potential}. Eventually, the last term in  the
right-hand side of 
\eqref{micromagenergy} is the {\it Zeeman energy}, namely
 the work done by  the external magnetic field $\bar H \in
L^1(\O_h;\R^3)$.  We anticipate that in Section \ref{evo} the
external field will  
depend on time and drive the  quasistatic  evolution of the film.

 An application of  the Direct Method of the Calculus of
Variations,  see e.g.~\cite{james-kinderlehrer},  ensures 
that for every $h>0$,  the micromagnetic energy $\bar E^{\rm
  mag}_h$ admits a minimizer  in   the set 
$$\MM_h:=\{\bar m\in W^{1,2}(\O_h;\R^3) \ : \   |\bar m | = \msat\mbox{ a.e.}\}.$$ 

 \subsection{Magnetomechanics} 

 The medium will be subject to nonhomogeneous time-dependent Dirichlet boundary
conditions on some distinguished part $\Gamma_h:= \omega \times (0,h)$ of the boundary
$\partial S \times (0,h)$ where $\omega \subset \partial S $ is of positive surface measure. In order to
 prescribe  these conditions we assume to be given $ \bar u^{\rm Dir}_h\in
W^{1,2}(\O_h;\R^3)$ and let $$\bar u + \bar u^{\rm
  Dir}_h:\O_h\to\R^3$$ be the displacement of the specimen from its
reference configuration.  We classically denote by
$\varepsilon(\bar u)$ the symmetrized gradient
$\varepsilon(\bar u):=(\nabla \bar u+\nabla^\top \bar u)/2$.  
Within the small deformation realm,  we linearly decompose the strain of the
material as
$$ \varepsilon(\bar u {+} \bar u^{\rm
  Dir}_h) = \varepsilon^{\rm elas} + \varepsilon^{\rm mag}(\bar
m).$$
Here, $\varepsilon^{\rm elas} $ is the elastic part of the strain. In
particular, $\varepsilon^{\rm elas} = \mathbb{C}^{-1}\sigma$, where  $\mathbb{C}$ is the elasticity tensor (symmetric,
positive definite) and $\sigma$ is the {\it stress} experienced by the
material. On the other hand, 
$\varepsilon^{\rm mag}(\bar
m)$ is the stress-free strain corresponding to the magnetization $\bar
m$. In particular, we could choose 
$$\varepsilon^{\rm mag}(\bar m):=\bar m\otimes \bar m-\frac{\msat^2}{3}\mathbb{I},$$
where $\mathbb{I}$ is the identity matrix in $\R^{3\times 3}$.  
Note 
that $\varepsilon^{\rm mag}$ is a symmetric, continuous, even, and 
deviatoric (as $|\bar m|=\msat$)   tensor-valued mapping of $\bar m$.  The specific
form of $\varepsilon^{\rm mag}$ is here chosen for definiteness only.  In
fact, other forms of $\varepsilon^{\rm mag}$ can also be covered by our model
as long as they  enjoy  the mentioned properties. 

The {\it elastic energy} of the
medium is classically described  by the quadratic  form  
\begin{align}
\bar E^{\rm elas}_h(t,\bar u,\bar
m)& :=\frac{1}{2|S|h}\int_{\O_h}\mathbb{C} \varepsilon^{\rm elas}{:}
\varepsilon^{\rm elas} \, \md x\nonumber\\
& =\frac{1}{2|S|h}\int_{\O_h}\mathbb{C}\big(\varepsilon(\bar
u{+}\bar u^{\rm Dir}_h)-\varepsilon^{\rm mag}(\bar m) \big){:}\big (\varepsilon(\bar u{+}\bar u^{\rm Dir}_h)-\varepsilon^{\rm mag}(\bar m) \big)\,\md x.\nonumber
\end{align}
 Given the magnetization $\bar m$, the elastic equilibrium problem
consists in finding $\bar u$ minimizing the elastic energy 
$\bar E^{\rm elas}_h$ %$E^{\rm elas}_h$ 
on the set of admissible displacements
\begin{align}
 \UU_h:=\{u\in W^{1,2}(\O_h;\R^3) \ : \ u=0 \mbox{ on }\Gamma_h\}.\nonumber
\end{align}    
 This problem has clearly a unique solution which depends linearly
on both $\epsi(\bar u^{\rm Dir}_h)$ and $\varepsilon^{\rm mag}(\bar m)$. Note that, in
particular, $\bar u+\bar u^{\rm Dir}_h= \bar
u^{\rm Dir}_h$ on $\Gamma_h$. 
 
The total magnetoelastic energy of  the specimen results from the
sum of the micromagnetic and the elastic energy and  reads
\begin{align}
\bar E_h(\bar u,\bar m):=\bar E^{\rm mag}_h(\bar m) +
\bar E^{\rm elas}_h(\bar u,\bar m). %E^{\rm elas}_h(\bar u,\bar m). 
\nonumber
\end{align}
 It is a rather standard matter to check that the total energy
admits minimizers $(\bar u ,\bar m)$ in the set $\UU_h \times
\MM_h$. These correspond to a variational solution of the
magnetoelastic system
\begin{subequations}\label{sys}
  \begin{align}
    \nabla {\cdot} \sigma &=0 \quad \text{in} \  \Omega_h,\\
    \mathbb{C}(\epsi(\bar u {+} \bar u^{\rm Dir}_h)-\varepsilon^{\rm
      mag}(\bar m)) &=\sigma \quad \text{in} \  \Omega_h,\\
    - \alpha \Delta \bar m + \nabla_{\bar m} \bar \varphi_h(x,
    \bar m) +\frac12 \nabla
    \bar
    \xi &=\bar H \quad \text{in} \  \Omega_h,\label{m}\\
    \nabla\cdot(-\mu_0\nabla \bar \xi+\bar m\chi_{ \O_h})&=0\quad
    \text{in} \  \R^3,\\
     \alpha \partial_{\nu} \bar m &=0 \quad \text{on} \  \partial \Omega_h,\\
    \bar u&=0 \quad \text{on} \ \Gamma_h
  \end{align}
\end{subequations}
where we have denoted by $\nu$ the outer unit normal to $\partial \Omega_h$.

%It will be convenient to denote $q:=(u,m)\in \UU_h\times \MM_h$, so that with a slight abuse of the notation we can write 
%$\mathcal{D}^{(h)}(q^1,q^2):= \mathcal{D}^{(h)}(m^1,m^2)$ if $q^i:=(u^i,m^i)$ for $i=1,2$ and so on. 

\section{ Static thin-film limits }\label{dimension-reduction}

 We shall preliminarily record here some dimension reduction analysis in the static situation of Section \ref{description}. Our aim is to investigate the limit $h \to 0$,
corresponding indeed to the situation of a very thin structure in the
$e_3$ direction. The aim of the section is to present some
corresponding $\Gamma$-convergence analysis. We discuss the micromagnetic and the magnetoelastic
limit separately.

\subsection{ Micromagnetic limit}
We  shall prove  the convergence of minimizers  
 of $\bar E^{\rm mag}_h$ %of $E^{\rm mag}_h$ for $H=0$ 
to minimizers of some limiting
energy $E^{\rm mag}_0$ as $h\to 0$. Our argument corresponds to an
extension of the analysis by {\sc Gioia \& James}
\cite{goia-james}, who investigated  the case of a thickness- and space-independent
 magnetic potential $\bar \varphi$ in absence of external field,
 i.e.  $\bar H=0$. We shall set  the result 
 within  the  classical 
$\Gamma$-convergence frame \cite{braides,dalmaso}. Considering  the 
standard rescaling  with $Z_h:={\rm diag}(1,1,1/h)$ and the mapping
$x\mapsto Z_hx$, we associate to $\bar m:\O_h\to\R^3$ a magnetization
$m:\O:=\O_1\to\R^3$, to $\bar \xi:\R^3\to \R$ the
rescaled magnetostatic potential   $\xi:\R^3\to \R$, 
to $\bar H:\O_h\to\R^3$ the external field $ H:\O\to\R^3$, 
and to $\bar \varphi_h: \Omega_h\times \msat S^2\to [0,\infty) $ the
rescaled magnetic potential $\varphi_h: \Omega\times \msat S^2 \to [0,\infty) $
defined as  
\begin{align*}
&m(Z_hx):=\bar m(x), \ \ H(Z_hx):=\bar H(x), \ \ \varphi_h(Z_hx,m):=
\bar \varphi_h(x,m),  \\ 
&\qquad \mbox{ and} \ \ \xi(Z_hy):=\bar \xi(y)\  \ \forall
x\in\O_h, \  y\in\R^3, \ m \in \msat S^2.
\end{align*}\EEE
 Correspondingly, we define the set $\MM : = \MM_1$.   By using the summation convention we can express  
$$
\nabla\bar m(x)=m,_{i}(Z_hx)\otimes e_i+ \frac1h m,_{3}(Z_hx)\otimes e_{3}\ \text{ where }i=1,2.$$
Moreover, we define  the {\it planar}  components of the magnetization and of the gradients as 
\begin{align*}
&m_p:=m_ie_i=(m_1,m_2), \  \ \nabla_p m:=m,_i\otimes e_i=(m,_1,m,_2), \\
&\nabla_p \xi:= \xi,_ie_i=(\xi,_1,\xi,_2).
\end{align*}  

 By exploiting   this rescaling and notation, we can 
equivalently  write the energy $\bar E^{\rm mag}_h(\bar m)$  in terms of $m$ as  $\bar E^{\rm mag}_h(\bar m)=E^{\rm mag}_h(m)$ where  
\begin{align*}
&E^{\rm mag}_h(m):= \frac{\alpha}{|S|}  \int_\O\left(|\nabla_p m(z)|^2 +
  \frac{1}{h^2}|m,_{3}(z)|^2\right)\, \md z \\
&+ \frac{1}{|S|} \int_\O\left(\varphi_h(z, m(z))-H(z)\cdot m(z) +\frac{1}{2}\left(\nabla_p \xi(z)\cdot m_p(z)+\frac1h \xi,_3(z)m_3(z)\right)\right)\,\md z\        
\end{align*}
where the relationship between the magnetization and the stray field
is given by the  Maxwell  equation in the whole space
\begin{align}\label{stray-scaled}
\nabla_p\cdot(-\nabla_p \xi + m_p\chi_\O)+\frac{1}{h}\frac{\partial}{\partial z_3}\left(-\frac1h \xi,_3+m_3\chi_\O\right)=0 .
\end{align}
Moreover, $m$  will be required to satisfy the saturation
constraint  $|m|=\msat$ a.e.~in $\O$. 
The following result can be found in \cite[Prop.~4.1]{goia-james}.

\begin{lemma}\label{gioia}
Let $\haz m_h\chi_\O\to \widetilde m\chi_\O$ in $L^2(\R^3;\R^3)$ as $h\to
0$, let $|\haz m_h|=\msat$ a.e.~in $\O$ and let $\haz\xi_h$ be the
solution to \eqref{stray-scaled} corresponding to $\widetilde m_h\chi_\O$.
Then, we have that $\|\nabla\haz\xi_h\|_{L^2(\R^3;\R^3)}\to 0$ and $\|h^{-1}\haz\xi_{h,3}-\widetilde m,_3\|_{L^2(\R^3)}\to 0$.
Moreover,  
\begin{equation}
\lim_{h\to 0}  \frac12  \int_\O\left(\nabla_p \haz\xi_h(z)\cdot (\haz
  m_{h})_p(z)+\frac1h \haz\xi_{h,3}(z)\,
   \haz m_{h3}(z) 
% \haz m_{h3}(z) 
  \right)\,\md
z= \frac12 \int_\O(\widetilde m_3(z))^2\,\md z .\label{eq_lemma}
\end{equation}
\end{lemma}

The next result describes the limiting micromagnetic energy. It can be
 basically 
found in \cite[Thm.~4.1]{goia-james}  although for   $H=0$. The 
extension  to  $H\ne 0$ is straightforward since the 
Zeeman  term is linear in the magnetization.  We shall use
the notation  $\haz H(z_1,z_2):= \int_0^1 H(z_1,z_2,s)\md s$  
for  $(z_1,z_2)\in S$  (note however that in  most applications  the
external field can be considered to be constant in~$\O$).  With
respect to \cite{goia-james}, we present here a rephrasing of the
result in terms of $\Gamma$-convergence of the micromagnetic energies
\cite{braides,dalmaso}.  In particular, we
need to be postulating some limiting behavior of the sequence
$\varphi_h$. Instead of heading to maximal generality,  
we prefer to present here a specific yet relevant case by assuming the decomposition 
\begin{equation}
\varphi_h(z,m) = f(h)\varphi_p(z_p,m_p) +
\varphi_3(z,m)\label{ansatz}
\end{equation}
where $\varphi_p: S \times \{m \in \R^2 \ | \ |m| \leq \msat\} \to [0,\infty)$ and $\varphi_3:
\Omega \times \msat S^2 \to [0,\infty)$ are continuous and $f: (0,1) \to [0,\infty)$ decreases.
Under suitable coercivity assumptions on $\varphi_p$, the  first term
in the above right-hand side penalizes the planar components of
$m$. This would ideally correspond to the observed behavior of some
ultrathin films showing a very strong magnetic anisotropy in the off-plane
direction \cite{Bruno88,Schulz94}.  

Along with the above structural {\it Ansatz} \eqref{ansatz}, the magnetic potential 
 $\varphi_h$ converges pointwise and monotonically to
$$\varphi_0(z,m) = f(0+)\varphi_p(z_p,m_p) + \varphi_3(z,m)$$
where the term $f(0+)\varphi_p$ has to be intended as the
constraint $\{\varphi_p=0\}$ in case $f(0+) = \infty$. Along with this
convention, we have
the $\Gamma$-convergence of the magnetic energy terms in terms of
the strong topology of $L^2$. 

\begin{proposition}[$\Gamma$-convergence of the micromagnetic energies]\label{gamma1}
 $E^{\rm mag}_h$ $\Gamma$-converges strongly in $L^2$ to $E^{\rm
  mag}_0$ given by 
$$
E^{\rm mag}_0(m):=
\begin{cases}
E^{\rm mag}_p(m)&:=\displaystyle\frac{1}{|S|}
\int_S\left(\alpha| \nabla_p  m|^2+ \varphi_0(z, \EEE m)- \haz H\cdot
  m+\frac12m_3^2\right)\,\md z_1\md z_2 \\
&\hspace{8mm}\mbox{if  $m\in W^{1,2}(\O;\R^3)$, $|m|=m_{\rm sat}$, and $m,_3=0$},\\
+\infty &\mbox{otherwise.}
\end{cases}
$$
\end{proposition}

\begin{proof}
 The existence of a recovery sequence follows by pointwise
convergence. Let $m \in \MM$ with $m,_3=0$. Then, the
constant sequence $m_h = m$ satisfies $$\lim_{h\to 0}E^{\rm
  mag}_h(m)= E^{\rm mag}_0(m)$$ due to Lemma \ref{gioia}. If on the
contrary $m,_3\not =0$ then $E^{\rm
  mag}_h(m) \to \infty$.

Let now $m_h \in \MM$ converge strongly in $L^2$ to $m\in \MM$.
As we are interested in checking that $\liminf_h E^{\rm mag}_h(m_h)
\geq E^{\rm mag}_0(m)$ we may assume with no loss of generality that 
$\liminf_h E^{\rm mag}_h(m_h)<\infty$ or even that $ E^{\rm
  mag}_h(m_h)$ is uniformly bounded. This entails in particular that
$m,_3=0$. It is hence sufficient to use \eqref{eq_lemma} in order 
 to get the liminf inequality. %to pass to the liminf.  
\end{proof}

\subsection{ Magnetoelastic limit }
 Let us here consider the thin-film limit for the magnetoelastic problem. We shall rescale the magnetoelastic energy 
 $\bar E^{\rm elas}_h(\bar u,\bar m)$ %$E^{\rm elas}_h(\bar u,\bar m)$ 
  and obtain a magnetoelastic Kirchhoff-Love
plate theory. In particular,  for $x\in\Omega_h$ we let 
$$\bar u(x)=:Z_hu(Z_hx)= Z_hu(x_1,x_2,x_3/h)$$  
 so that $u:\Omega \to \R^3$. Correspondingly, we define the set $\UU : = \UU_1$.  
It follows that, for   $i,j=1,2,$ 
$$
\varepsilon_h(u):=Z_h\varepsilon (\bar u)Z_h= \left(\begin{array}{ccc}
\varepsilon(u)_{ij} & \displaystyle\frac{1}{h}\varepsilon(u)_{i3}\\[3mm]
\displaystyle\frac{1}{h}\varepsilon(u)_{i3} &\displaystyle\frac{1}{h^2}\varepsilon(u)_{33} \\
\end{array}\right).
$$
 Analogously the  scaling  for  $\varepsilon^{\rm mag}$ will be 
$$\varepsilon^{\rm mag}_h(m):=Z_h \varepsilon^{\rm mag}(\bar m(Z_h))Z_h,$$
so that, for $i,j = 1,2$,
$$
\varepsilon^{\rm mag}_h(m):=\left(\begin{array}{ccc}
(\varepsilon^{\rm mag}(m))_{ij} & \displaystyle\frac{1}{h}(\varepsilon^{\rm mag}( m))_{i3}\\[3mm]
\displaystyle\frac{1}{h}(\varepsilon^{\rm mag}(m))_{i3} & \displaystyle\frac{1}{h^2}(\varepsilon^{\rm mag}(m))_{33} \\
\end{array}\right).
$$
As to boundary conditions, we consider
\begin{align}\label{K-L-boundary}
 u^{\rm Dir}\in  \KK:=\{u \in  W^{1,2}(\O;\R^3)\ : \ \varepsilon(u)_{3i}=\varepsilon(u)_{i3}=0 \mbox{ for } i=1,2,3 \}\end{align}
and set  (analogously to the choice for $u$) 
$\bar u^{\rm Dir}_h(x):=Z_h  u^{\rm Dir}  (Z_hx)$ for  $x\in\O_h$. 
The  space $\KK$   in \eqref{K-L-boundary}  represents the
admissible displacements for  Kirchhoff-Love  plates. In particular, $u \in \KK$ entails 
$$ u_{1,3}+u_{3,1} = u_{2,3}+u_{3,2}=u_{3,3}=0.$$
Namely $u_3$ is constant in direction $e_3$ and $u_1, u_2$ are affine
in direction $e_3$.   

 We shall define the energy $E^{\rm elas}_h$ on the  rescaled 
 domain $\O$
 via $E^{\rm elas}_h(u, m)=\bar E^{\rm elas}_h(\bar u,\bar m)$. In
particular, we have 

\begin{align}
 E^{\rm elas}_h(u, m)
:=\frac{1}{2|S|}\int_{\O}\mathbb{C}\big(\varepsilon_h(u{+}u^{\rm Dir}){-}\varepsilon_h^{\rm mag}(m)\big){:}\big(\varepsilon_h(u{+}u^{\rm Dir}){-}\varepsilon_h^{\rm mag}(m)\big)\,\md z .
\end{align}

Let us denote  the set of admissible states by 
\begin{align}
\mathcal{Q}:= \left\{(u,m)\in\UU \times\MM \ : \
  \varepsilon(u)_{i3}=\varepsilon^{\rm mag}(m)_{i3} \ \text{for} \ i=1,2,3\right\}.\nonumber
\end{align}

Then,  $E^{\rm elas}_h $ admits a minimizer in
$\mathcal{Q}$. Moreover, the component $u$ of such minimizer depends
linearly on the component $m$. In particular, given $m$, the
displacement $u$ is uniquely
determined.

 In order to discuss the limiting case $h\to 0$, by  following
\cite{braides} or \cite{freddi-paroni-zanini}   we define, for
$i,j,k,\ell=1,2$,  the limiting elasticity tensor $\mathbb{C}^0$
as 
 \begin{align}
 \mathbb{C}^0_{ijk\ell}:=\mathbb{C}_{ijk\ell}-\frac{\mathbb{C}_{ij33}\mathbb{C}_{k\ell33}}{\mathbb{C}_{3333}} .
 \end{align}
 For all $ A\in\R^{2\times 2}$ we let the quadratic form $Q:\R^{2\times 2} \to [0,\infty)$ be defined as 
  $$Q(A):=\min_{a\in\R^{2\times  1 },b\in\R}\mathbb{C}\left(\begin{array}{ccc}
  A & a\\
a^\top & b \\
\end{array}\right): \left(\begin{array}{ccc}
  A & a\\
a^\top & b \\
\end{array}\right) .
$$
 One readily checks that the  minimum is achieved at
$$a=0 \ \ \text{and} \ \ b= - \frac{\mathbb{C}_{ij33} A_{ij}}{\mathbb{C}_{3333}} .$$
 In particular,  we have that
$Q(A)=\mathbb{C}^0 A{:} A$ for all $A\in \R^{2\times 2}$ and that $Q$ is
 uniformly convex on $\R^{2\times 2}$.  Let us use the notation $\varepsilon_p \in \R^{2\times 2}$ in order to indicate the {\it
  planar} block of the matrix $\varepsilon \in \R^{3\times 3}$, namely
$(\varepsilon_p)_{ij} = (\varepsilon)_{ij}$ for $i,j
=1,2$. We have the following.

\begin{proposition}[$\Gamma$-convergence of the magnetoelastic energies]\label{Prop-recovery}
For all $u^{\rm Dir} \in \KK$ we have that $E^{\rm elas}_h$ $\Gamma$-converges to $ {E}^{\rm elast}_0$ with
respect to the weak topology of $W^{1,2}$ where
$$
 {E}^{\rm elast}_0(u,m):=
\begin{cases}
 E^{\rm elas}_p(u,m) &:= \displaystyle\frac{1}{2|S|} \int_{
  \O }Q\big(\varepsilon_p(u){+}\varepsilon_p(u^{\rm Dir}(t)){-}\varepsilon_p^{\rm
  mag}(m)\big) \md z\\
 \qquad & \quad \mbox {if  $(u,m)\in \mathcal{Q}$},\\
+\infty &\quad \mbox{otherwise.}
\end{cases}
$$
\end{proposition}

\begin{proof}
 Let $(u_h,m_h)\to (u,m)$ weakly in $W^{1,2}$ and assume with no
loss of generality that $\liminf_{h\to 0}
E^{\rm elas}_h(u_h,m_h)<\infty$ or even, possibly extracting but not
relabelling,  that $E^{\rm elas}_h(u_h,m_h)$ are uniformly bounded. Then, since
$u^{\rm Dir} \in \KK$, the
limit $(u,m)$ belongs necessarily to $\mathcal{Q}$. Hence,  the 
definition of $Q$ and  its  lower semicontinuity  imply 
that  $\liminf_{h\to 0}
E^{\rm elas}_h(u_h,m_h)\ge {E}^{\rm elas}_p(u,m)$.  

On the other hand, consider $(u,m)\in\mathcal{Q}$.  Define  $b\in
L^2(\O)$  via 
\begin{align}\label{b}
b:= - \frac{\mathbb{C}_{ij33} (\varepsilon(u{+}u^{\rm Dir})_{ij}-\varepsilon^{\rm mag}(m)_{ij})}{\mathbb{C}_{3333}}\ \end{align}
and a sequence $\{\psi_h\}_{h>0}\subset C^\infty_0(\O)$ such that
$\psi_h\to b$ in $L^2(\O)$ and $h\nabla\psi_h\to 0$ in
$L^2(\O;\R^3)$. Let $\phi_h$ be such that
% $\partial\phi_h(x_1,x_2,x_3)/\partial x_3:=\varphi_h(x_1,x_2,x_3)$
% for all $(x_1,x_2)\in S$ and $x_3\in (0,1)$.
 $\phi_{h,3} := \psi_h$ in $\Omega$. 
Define  for $h>0$ and  i=1,2,3, 
$$
\widetilde u_{hi} :=
\begin{cases}
u_i & \mbox{ for $i =1,2$,}\\
u_i +h^2\phi_h & \mbox{ for $i=3$.}
\end{cases}
$$
Hence, 
\begin{align}
&\varepsilon_h(\widetilde u_h{+}u^{\rm Dir})-\varepsilon_h^{\rm
  mag}(m)\nonumber\\
&= \left(\begin{array}{cc}
  \varepsilon(\widetilde u_h{+}u^{\rm Dir})_{p}{-}\varepsilon^{\rm
    mag}(m)_{p} & \displaystyle\frac{(\varepsilon(\widetilde u_h{+}u^{\rm Dir})_{i3}{-}\varepsilon^{\rm mag}(m)_{i3})}{h}{+} \frac{h}{2}\frac{\partial\phi_h}{\partial x_i}\\
 \displaystyle\frac{(\varepsilon(\widetilde u_h{+}u^{\rm
   Dir})_{3i}{-}\varepsilon^{\rm mag}(m)_{3i})}{h}{+}\displaystyle \frac{h}{2}\frac{\partial\phi_h}{\partial x_i}&\!\!\!\! \psi_h\\
\end{array}\right).\nonumber
\end{align}
 As $u^{\rm Dir} \in \KK$, we have 
$\varepsilon(u{+}u^{\rm Dir})_{i3}-\varepsilon^{\rm mag}(m)_{i3}=0$.  In
particular, the terms with factor $1/h$  vanish.  Then,
passing to the $\liminf$ for $h \to 0$ we readily get that $E^{\rm
  elas}_h(\widetilde u_h, m) \to E^{\rm elas}_p(u,m)$. In particular,
$(\widetilde u_h, m) $ is a recovery sequence for $(u,m)$. 

In case $(u,m) \not \in \mathcal{Q}$, the same choice $(\widetilde u_h, m)
$ entails $E^{\rm elas}_h(\widetilde u_h, m)\to \infty$. Namely, 
$(\widetilde u_h, m)$ still provides a  recovery sequence.
\end{proof}
 
 The rescaled total magnetoelastic energy $E_h$ is defined
on the domain $\O$ via $E_h(u,m) = \bar E_h(\bar u, \bar m)$. In particular, we have
$$
   E_h(u,m) = E^{\rm mag}_h(m) + E^{\rm elas}_h(u,m).
$$

\section{  Quasistatic  evolution}\label{evo}

 Let us now turn to the analysis of the  quasistatic 
evolution case. The aim here is to introduce and analyze a rate-independent
model for  a  magnetoelastic Kirchhoff-Love plate. We obtain this by
dimension reduction, by passing to the limit in $h>0$ for a
three-dimensional magnetoelastic evolution model. In particular, we
detail in Subsection \ref{analysis} the  quasistatic  evolution problem
for the three-dimensional specimen and discuss in Subsection \ref{KL}
the evolutive thin-film limit  constituing  the plate model. 

\subsection{  Quasistatic  evolution in the bulk}\label{analysis}
\setcounter{equation}{0}

By possibly assuming the Dirichlet datum $\bar u^{\rm Dir}_h$ and/or
the external field  $H$ to change with time, the minimizers $(\bar
u,\bar m)$ of $\bar E_h$ evolve as well. In order to prescribe a
suitable evolution law, we postulate magnetic dissipation. In
particular, by assuming that the changes in the data are so slow that
inertial effects can be neglected, we assume that the (time-dependent)
state of the system $t \mapsto (\bar u(t),\bar m(t))$ solves relations \eqref{sys} on
$[0,T]$ (and in a
suitable variational sense, see below) where however the static
relation \eqref{m} is replaced by the rate-independent inclusion
$$ %R\partial |\bar m_t|
 \partial \psi(\bar m_t) - \alpha \Delta \bar m +  \nabla
_{\bar m}\bar \varphi_h(x,  \bar m) +\frac12 \nabla
    \bar
    \xi \ni \bar H \quad \text{in} \  \Omega_h \times (0,T).$$
The symbol $\partial$ above indicates the subdifferential in the sense
of convex analysis and  $\psi(\bar m_t)$ measures the infinitesimal
dissipation involved in the process.  As the thickness $h$
decreases, an additional magnetic anisotropy effect arises. While bulk materials $h=1$
show isotropic dissipation, in the thin-film limit $h \to 0$
anisotropic dissipation can be observed
 \cite{uno,due,tre}. In particular, in some regimes  the 
dissipation tends to be larger for processes involving off-planar
magnetizations. We shall take this into account by choosing 
$$\psi(\bar m_t) = R_p |\bar m_{p,t}| + R_3(h)|\bar m_{3,t}|.$$
Here, $R_p >0$ is an energetic yield
limit for evolution in the plane \cite{hubert-schaefer}  which we
assume to be independent of the film thickness, for simplicity.  On
the other hand, the function $h \mapsto R_3(h) >0$ models
anisotropic effects in
the $e_3$ direction which are 
observed to be thickness-dependent \cite{quattro}. \EEE We shall here limit ourselves in assuming that the right limit $R_3(0_+)$ exists and is
finite. Note nonetheless that the case
$R_3(0_+)=\infty$, imposing indeed $\bar m_{3,t}=0$, could be considered as well.
 The latter
equation corresponds to the postulate that the
energy released  by changing the state of the system
from $(\bar u^1,\bar m^1)$ to $(\bar u^2,\bar m^2)$   is given by the
simple form 
 \begin{align}
 \bar D_h(\bar m^1,\bar m^2):= \frac{1}{|S|h}\int_{\O_h} \left(R_p|\bar
 m^1_p{-}\bar m^2_p| + R_3(h)|\bar
 m^1_3{-}\bar m^2_3|\right)\md x.  \nonumber
 \end{align}
Note that the
dissipation $\bar D_h$  is positively $1$-homogeneous  and, correspondingly, the  evolution will be
rate-independent. In particular, energy will be dissipated by purely
hysteretic losses.

For the sake of later convenience, let us reformulate the problem in
the fixed reference configuration $\Omega$. This amounts in 
considering  the energies $ E_h^{\rm mag}$, $ E_h^{\rm elas}$, and $
E_h $ (here assumed to be depending on time as well, without
introducing new notation) and the dissipation 
$D_h(m^1,m^2)=\bar D_h(\bar m^1,\bar m^2)$  so that 
$$
 D_h( m^1, m^2)=\frac{1}{|S|}\int_{\O} \left(R_p|
 m^1_p{-}m^2_p| +  R_3(h)|
 m^1_3{-}m^2_3|\right)\md x.\EEE
$$
%shall we keep D_h in the following or shall we replace it by D?

Assume to be given time-dependent boundary datum $t \in [0,T]\mapsto \bar
u^{\rm Dir}(t) \in W^{1,2}(\O,\R^3)$ and 
external field $t \in [0,T]\mapsto  H(t) \in L^1(\O_h;\R^3)$. We are
interested in proving the existence of a  quasistatic  evolution $t\in [0,T] \mapsto ( u_h(t),m_h(t)) \in
\mathcal{Q} $ 
in the form of  the
so-called {\em energetic formulation} \cite{mielketheil}.  Given
some suitable initial datum 
 $(u^0, m^0) \in\mathcal{Q}$ %$(\bar u^0,\bar m^0) \in\mathcal{Q} $ 
we define it as follows. 

\begin{definition}[Energetic solution in the bulk]\label{Def:ES-bulk} 
 An \emph{energetic solution} of the  quasistatic   evolution  in
the bulk
 is a  trajectory  $t\in[0,T]\mapsto (  u_h(t),  m_h(t))
 \in \mathcal{Q} $ such that $(  u_h(0),  m_h(0))=(  u^0,  m^0)$ and, for every $t\in [0,T]$,
\begin{align}
&  E_h(t,  u_h(t),  m_h(t))\leq   E_h(t,\haz u,\haz m)+ 
D_h(  m_h(t),\haz m)\quad \forall\,  (\haz u,\haz m)\in  \mathcal{Q} \tag{S}\label{S}\\[2mm]
&  E_h(t,  u_h(t),  m_h(t))+\hbox{\rm Diss}_{  D_h}(  m_h,[0,t])
\nonumber\\
&\qquad =  E_h(0,  u^0,  m^0) +\int_0^t \partial_t E_h(s,  u_h(s),  m_h(s))\ds\tag{E}\label{E}
\end{align}
where  $\hbox{\rm Diss}_{  D_h}(  m_h,[0,t])$ is the \emph{total dissipation} on $[0,t]$ defined by
\begin{equation}\label{dissi}\hbox{\rm Diss}_{  D_h}(  m_h,[0,t]):=\sup\left\{\sum_{i=1}^{N}  D_h(  m_h(t^i),  m_h(t^{i-1})) \right\},
\end{equation}
the supremum being taken over all  partitions $ \{0=t^0 <t^1<
\ldots<t^N=t\}$ of $[0,t]$.
\end{definition}
 The two conditions \eqref{S}-\eqref{E} in the definition of
energetic solution have an immediate
mechanical interpretation. Condition \eqref{S} is a {\it global
  stability} criterion: Transitions from the actual state $( 
u(t),  m(t))$ to some possible competitor state $(\haz u,\haz m)$ is not
energetically favored in the sense that  the  energy gain is
compensated by the  dissipation cost.  
For later  notational  convenience, we define the set of {\it stable states} at time $t\in [0,T]$ as
\begin{align*}
\Stbl_h(t)&:=\Big\{(  u_h,  m_h)\in  
\mathcal{Q}: \  %\mathcal{Q}_h :
  E_h(t,  u_h,  m_h)\leq  E_h(t,\haz u,\haz m){+}  D_h( 
m_h,\haz m),\  \forall\,(\haz u,\haz m)\in \mathcal{Q} %\mathcal{Q}_h 
\Big\}
\end{align*}
 so that condition \eqref{S} equivalently reads $( 
u_h(t),  m_h(t))\in \Stbl_h(t)$ for all $t \in [0,T]$. The scalar
equation \eqref{E} is nothing but {\it energy conservation}: It
expresses the balance between current and dissipated energy
 (left-hand side) %(right-hand side) 
and initial energy plus work of external actions
 (right-hand side). %(left-hand side)

 Let us close this section by recording an   existence result
 for  quasistatic  evolutions in three dimensions.

\begin{theorem}[Existence for the  quasistatic  evolution in
 the bulk]\label{thcv2} Let $h>0$. Assume to be given  $H\in  C^1([0,T];  L^1(\O;\R^3))$,  $  u^{\rm
   Dir}_h\in  C^1([0,T];  W^{1,2}(\O;\R^3))$,  and 
 $(  u^0,  m^0)\in\Stbl_h(0)$. Then,  there exists an energetic
 solution $(  u_h,  m_h)$ for the  quasistatic  evolution problem.
\end{theorem}

We shall not report here a  proof of Theorem \ref{thcv2} as it may be
readily obtained in the frame of the by now classical existence theory
for energetic solutions by {\sc Mielke \& Theil}
\cite{mielke01,mielketheil}. Indeed, it is sufficient to point out
that  $E_h$
has bounded (hence weakly compact) sublevels in $\mathcal{Q}$,
that $D_h$ is continuous with respect to the same topology, and that
the power $\partial_t E_h$ is well behaved in
order to apply, for instance, \cite[Thm. 5.2]{mielke01}.

\subsection{ Quasistatic  evolution of the magnetoelastic thin-film}\label{KL}

Let us now come to the description of the magnetoelastic thin-film,
which results in a
Kirchhoff-Love plate model. We shall derive this by taking the limit
$h\to 0$ in the three-dimensional evolution model. The state of the material will be
described by the pair 
\begin{align*}
(u,m)&\in \QQQ :=  \{ (u,m) \in \mathcal{Q} \ \ 
 \text{such that} \  m_{{,3}}=0\}\\
& = \{(u,m) \in W^{1,2}_0(\O,\R^3)\times W^{1,2} (\O,\R^3) \
 \text{such that }\\
&\hspace{5mm}\varepsilon_{i3}(u) =  \varepsilon^{\rm mag}(m)_{i3} \ \text{for} \ i=1,2,3, \ |m|=\msat, \ m,_3=0 \
\text{a.e.}
\}
\end{align*}
 and its statics will
correspond to the minimization of the thin-limit energies of Section
\ref{dimension-reduction}.  
 In the following, the boundary datum and the external
field are time-dependent and will be driving the  quasistatic 
evolution of the medium. Correspondingly, we will indicate
time-dependence in the total energy of the medium  as
$ E_0(t,u,m) = E^{\rm mag}_0(t,m) + E^{\rm elas}_0 (t,u,m)$. Note that
$E_0(t,\cdot)$  is finite on $\QQQ$.

As for the dissipation, % we preliminary set 
% 
%\begin{align}\label{dissipation-h} D_h(m^1,m^2):=  \frac{1}{|S|}
%  \int_\O R|m^1(z)-m^2(z)|\,\md z\ \ (=   D(  m^1,  m^2))
%  \end{align}
%  
%\begin{align}\label{dissipation-h} D_h(m^1,m^2):=  \frac{1}{|S|h}
%  \int_\O d|m^1(z)-m^2(z)|\,\md z\ \ (=   D_h(  m^1,  m^2))
%  \end{align}
%so that, 
for all $m^1,m^2\in\mathcal{M}$ which are hence constant in
the direction $e_3$, we  define 
\begin{align}\label{dissipation}
D_0(m^1,m^2):=  \frac{1}{|S|} \int_S  \left(R_p%d
|m^1_p(z){-}m^2_p(z)| + R_3(0_+)|m^1_3{-}m^2_3|\right)\md z . \end{align}
 
 Owing to these definitions, the  quasistatic  evolution problem
for the magnetoelastic thin film can be reformulated in
terms of a rate-independent evolution driven by the potentials
$( E_0,D_0)$. As before, we shall be interested in energetic
solutions.

\begin{definition}[Energetic solution for the thin film]\label{Def:ES-bulk2} 
 An \emph{energetic solution} of the  quasistatic   evolution  for
 the magnetoelastic thin film
 is a  trajectory  $t\in[0,T]\mapsto ( u(t), m(t))  \in
\QQQ  $ such that $( u(0), m(0))=( u^0, m^0)$ and, for every $t\in [0,T]$,
\begin{align}
& E_0(t, u(t), m(t))\leq  E_0(t,\haz u,\haz m)+ D_0( m(t),\haz m)\quad
\forall\,  (\haz u,\haz m)\in \QQQ \tag{S2}\label{S2}\\[2mm]
& E_0(t, u(t), m(t))+\hbox{\rm Diss}_{ D_0}( m,[0,t])\nonumber\\
&\qquad
= E_0(0, u^0, m^0) +\int_0^t \partial_t E_0(s, u(s), m(s))\ds\tag{E2}\label{E2}
\end{align}
where  $\hbox{\rm Diss}_{ D_0}( m,[0,t])$ is the \emph{total
  dissipation} on $[0,t]$ defined analogously to
$\hbox{\rm Diss}_{  D_h}$, but starting from the
dissipation $D_0$. 
\end{definition}

 Let us denote by  $\Stbl_0(t)$ %The symbol $\Stbl_0(t)$ stands for 
the set of
stable states at time $t$, namely  of %for 
pairs $(u,m) \in \QQQ$ fulfilling \eqref{S2}.

We shall now prove that  energetic solutions to the   quasistatic 
evolution problem for the 
magnetoelastic Kirchhof-Love plate exist. Indeed, our result is 
stronger, as we prove that sequences of solution of the bulk model
admit subsequences which converge to energetic solution of the
thin-film model. In particular, we provide an approximation
result based on dimension reduction.

\begin{theorem}[Convergence to the thin film]\label{thm}
Let $u^{\rm Dir}\in C^1([0,T];W^{1,2}(\O;\R^3)$, $H\in
C^1([0,T]; L^1(\R^3))$,  $ (u^0,m^0)\in \Stbl_0(0)$, $ E_h(0,u^0,m^0)\to  E_0(0,u^0,m^0)$,  and 
$\{(  u_h,  m_h)\}_{h>0}\subset\mathcal{Q}_h$ be  a sequence of
energetic solutions  of the  quasistatic  evolution in three dimensions, i.e. solving \eqref{S}-\eqref{E}. Then, for some not relabeled
subsequence we have that $(  u_h,  m_h)\rightharpoonup (u,m)$ in
$W^{1,2}(\O;\R^3)\times W^{1,2}(\O;\R^3)$ where $(u,m)$ is an
energetic solution for the plate, i.e. solving\eqref{S2}-\eqref{E2}.
\end{theorem}       

In order to show that  an energetic solution of the bulk material converges
to an energetic solution to a plate we  apply the abstract
strategy introduced  in
\cite{mrs}.  We shall not provide here a detailed proof, but
rather comment on two crucial points of the argument. The first of
these points concerns functional convergence. In particular, we shall
establish a specific {\it evolutive
  $\Gamma$-convergence} notion, adapted to rate-independent
evolutions. Indeed, the theory relies on the verification of
{\it two separate} $\Gliminf$ inequalities
\begin{equation}
  \label{gliminf}
   E_0 \leq \Gliminf_{h \to 0}  E_h, \qquad D_0 \leq \Gliminf_{h
    \to 0} D_h
\end{equation}
as well  as on a {\it mutual recovery sequence} condition. The
$\Gliminf$ inequality for $ E_h$ follows by easily adapting   the results of Section
\ref{dimension-reduction} to the present time-dependent case.  On
the other hand, the $\Gliminf$ inequality for $D_h$ is immediate as
$R_3(h) \to R_3(0_+) \geq 0$. 
The
following lemma  entails the existence of a  mutual recovery sequence.

\begin{lemma}[Mutual recovery sequence]\label{joint}
Let  $(t_h,u_h,m_h) \rightharpoonup (t,u,m)$ in $[0,T]\times
\mathcal{Q}$, and $(\haz u,\haz m)\in\QQQ$.  Then, there  exist $(\haz u_h,\haz m_h)
\rightharpoonup (\haz u,\haz m) $ such that 
\begin{align}
\limsup_{h\to 0} ( E_h(t_h, \haz u_h, \haz m_h) +D_h(m_h,\haz
m_h))\le  E_0(t,\haz u,\haz m)+D_0(m,\haz m).
\end{align} 
\end{lemma}

\begin{proof}
 For all $h>0$, we  choose  $\haz
m_h:=\haz m$  and $\haz  u_h$  as in the proof of
Proposition~\ref{Prop-recovery}. The claim  then follows by the
 continuous convergence of $D_h$ to $D_0$ 
%continuity of $D_h$%$D$ 
with respect to the strong $L^2$-convergence of its
arguments.
\end{proof}

 A second crucial point for the possible application of the
abstract argument of \cite{mrs} consists in the convergence proof of
the  power of the energy  functionals. We shall argue here
in the same spirit of \cite{freddi-paroni-zanini}. 

\begin{lemma}\label{powers}
Let $u^{\rm Dir}\in C^1([0,T];W^{1,2}(\O;\R^3))$ and $H\in
C^1([0,T];L^1(\R^3))$. 
Let $(t,u,m)\in(0,T)\times\QQ$ and  assume  that there is a
sequence $(t_h,u_h,m_h)\in (0,T)\times \QQ$ such that
$(u_h,m_h)\in\mathcal{S}_h(t_h)$ and $t_h\to t$, $u_h\rightharpoonup
u$ and $m_h\rightharpoonup m$ in $W^{1,2}(\O;\R^3)\times
W^{1,2}(\O;\R^3)$. Then,  we have the convergence of the energies
and the  corresponding  powers
\begin{align*}
  &E_h(t_h,u_h,m_h)\to  E_0(t,u,m),\\
&\partial_t E_h(t_h,u_h,m_h)\to\partial_t E_0(t,u,m).
\end{align*} 
\end{lemma}

\begin{proof}
 Let us first show the convergence of the energies. The inequality $$\liminf_{h\to
  0} E_h(t_h,u_h,m_h)\ge  E_0(t,u,m)$$  follows from the 
already checked $\Gamma$-convergence of the functionals,
i.e. Propositions~\ref{gamma1} and \ref{Prop-recovery}. To check for  the other inequality we use Lemma~\ref{joint} and stability. 
Indeed,  by  choosing   $\haz u:=u$, $\haz m:=m$ and 
letting $(\haz u_h, \haz m_h)$ be the corresponding sequences from Lemma~\ref{joint}  we have 
\begin{align*}
  \limsup_{h\to 0} E_h(t_h,u_h,m_h)&\le \limsup_{h\to 0}( E_h(t_h, \haz u_h,
  \haz m_h) +D_h(m_h,\haz m_h))  \\
&\le E_0(t,\haz u,\haz m)  + D_0(m,m)=  E_0(t, u, m),
\end{align*}
where the first inequality follows from the stability  $(u_h,m_h)
\in \Stbl_h(t_h)$. 

 Let us now compute the power of $ E_h$ as  
\begin{align*}
\partial_t
E_h(t_h,u_h,m_h)&=\int_{\Omega_h}\mathbb{C}(\varepsilon(u_h){+}\varepsilon(u^{\rm
  Dir}(t_h)){-}\varepsilon^{\rm mag}(m_h)):\varepsilon(\dot u^{\rm
  Dir}(t_h))\, \md x\\
&-\int_{\Omega_h} \dot H(t_h)\cdot m_h\,\md x.
\end{align*}
An analogous expression holds for $\partial_t E_0(t,u,m)$. The
convergence of the first term in the expression of $\partial_t
E_h(t_h,u_h,m_h)$ can be proved as in \cite{freddi-paroni-zanini}
while the convergence of the second term %describing
is immediate by linearity.
\end{proof}

 Given the $\Gamma$-convergence of the functionals (Section~\ref{dimension-reduction}) and the powers (Lemma~\ref{powers}) and the
existence of a mutual recovery sequence (Lemma \ref{joint}), it
suffices to remark that $ E_h$ is coercive with respect to the 
weak topology of $W^{1,2}(\O;\R^3)\times W^{1,2}(\O;\R^3)$  in
order to obtain Theorem \ref{thm} by applying the abstract theorem  \cite[Thm~3.1]{mrs}.

Before closing this discussion let us explicitly note that the
developed technologies would allow also to deduce additional
dimension reduction results. In particular, by neglecting mechanical
effects, one could consider the possibility of deducing a
rate-independent model for the  quasistatic  evolution of a thin-film
driven by micromagnetic energy. This would constitute an evolutive
counterpart to the static analysis in \cite{goia-james}.

\section*{Acknowledgments} This work was initiated during a visit of
MK in IMATI CNR Pavia. Its hospitality and support is gratefully
acknowledged. US is partially supported by the CNR-JSPS grant {\it VarEvol}.

\medskip
% The data information below will be filled by AIMS editorial staff
Received xxxx 20xx; revised xxxx 20xx.
\medskip

\end{document}